\documentclass{pspum-l}
\usepackage{amsmath, amsthm, amscd, amsfonts,amssymb} 

\numberwithin{equation}{section}

\newcommand{\teq}{\arabic{section}.\arabic{equation}}
\newcommand{\teql}{\Alph{section}.\arabic{equation}}

\newcommand{\sqr}[2]{{\vcenter{\vbox{\hrule height.#2pt\hbox{\vrule width.#2pt
height#1pt \kern#1pt\vrule width.#2pt}\hrule height.#2pt}}}}

\newcounter{eqcount}

\newenvironment{edesc}{\refstepcounter{equation}\begin{enumerate}}%
{\end{enumerate}}
\newenvironment{triv}{\refstepcounter{equation}\begin{list}%
{{\hbox{\rm(\teq)\ }}} \item }{\end{list}}

\newcommand{\ring}[1]{{\mathbb #1}}
\newcommand\bZ{{\ring{Z}}}

\newcommand\bC{{\ring{C}}} \newcommand\bR{{\ring{R}}}
\newcommand\bF{{\ring{F}}} \newcommand\bQ{{\ring{Q}}}

\newcommand{\csp}[1]{{\mathbb #1}}
\newcommand{\tsp}[1]{{\mathcal #1}}

\newcommand{\prP}{\csp{P}}

\newcommand{\sO}{{\tsp{O}}} 
\newcommand{\sQ}{\tsp{Q}}

 \newcommand{\sH}{{\tsp {H}}}

\newcommand{\eql}[2]{{\rm (\ref{#1}\ref{#2})}} 

\newcommand{\vect}[1]{{\pmb #1}} 
 \newcommand{\bg}{\vect{g}}
 
\newcommand{\bp}{{\vect{p}}}

\newcommand{\row}[2]{{#1_1,\ldots,#1_{#2}}}

\newcommand{\smatrix}[4]{{\big(\begin{array}{cc}
\!\lower2pt\hbox{$\scriptstyle#1$} &\lower2pt\hbox{$\scriptstyle#2$}\!
\\\! \raise2pt\hbox{$\scriptstyle#3$} &\raise2pt\hbox{$\scriptstyle#4$}
\!\end{array}\big)}}

\newcommand{\texto}[1]{{\textr{#1}}}
 \newcommand{\SL}{\texto{SL}}
 \newcommand{\ind}{\texto{ind}}
\newcommand{\PSL}{\texto{PSL}} \newcommand{\PGL}{\texto{PGL}}
 
\newcommand{\Hom}{\texto{Hom}} \renewcommand{\ni}{\texto{Ni}}
 \newcommand{\Pic}{\texto{Pic}}
\newcommand{\textr}[1]{{\text{\rm #1}}}
 \newcommand{\ord}{\textr{ord}}

 \newcommand{\inn}{\textr{in}}

\newcommand{\rd}{\texto{rd}}

\newcommand{\tG}[1]{{}_{#1}\tilde G}

\newcommand{\ot} {\otimes}

\newcommand{\textb}[1]{{\text{\bf #1}}}
\newcommand{\bfC}{{\textb{C}}}
\newcommand{\longmapright}[2]{\smash{\mathop{\hbox to
#2pt{\rightarrowfill}}\limits^{#1}}}
\newcommand{\longmapleft}[2]{\smash{\mathop{\hbox to
#2pt{\leftarrowfill}}\limits^{#1}}}

\newcommand{\mapright}[1]{\smash{\mathop{\longrightarrow}\limits^{#1}}}

   \newcommand{\nm}{{-}}

\newcommand{\lrang}[1]{{\langle #1\rangle}}

\newcommand{\eqdef}{\stackrel{\text{\rm def}}{=}}

\newfont{\sevenrm}{cmr7}
\newfont{\bsevenrm}{cmbx7}
\newfont{\mathseven}{cmsy7}
\newfont{\bigmath}{cmsy10 scaled 1200}
\newfont{\fiverm}{cmr5}
\newfont{\bfiverm}{cmbx5}
\newfont{\hel}{cmbx10 scaled 1200}
\newfont{\eu}{eufb10}
\newfont{\sseu}{eufm5}
\newfont{\seu}{eufm7}
\newfont{\Cal}{cmmib10}
\newfont{\sCal}{cmmib7}
\newfont{\zch}{eusb10}

\theoremstyle{plain}
\newtheorem{thm}{Theorem}[section] 
\newtheorem{lem}[thm]{Lemma}

\newtheorem{prop}[thm]{Proposition}
\newtheorem{cor}[thm]{Corollary}

\theoremstyle{definition}
\newtheorem{defn}[thm]{Definition}
\newtheorem{exmp}[thm]{Example}
\newtheorem{guess}[thm]{Conjecture}

\newtheorem{prob}[thm]{Problem}

\theoremstyle{remark}
\newtheorem{rem}[thm]{Remark}

\newcommand{\xs}{\times^s\!}

\newcommand{\one}{{\pmb 1}} 
\newcommand{\Ind}{{\text{\rm Ind}}} 
\newcommand{\rank}{{\text{\rm Rank}}} 
\newcommand{\inv}{{\text{\rm inv}}} 
\newcommand{\aug}{{\text{\rm aug}}} 
\newcommand{\soc}{{\text{\rm Soc}}} 
\newcommand{\rad}{{\text{\rm Rad}}} 
\newcommand{\C}{{\text{\rm C}}}

\newcommand{\sh}{{\text{\bf sh}}}
\renewcommand{\phi}{{\varphi}}
\newcommand{\mpr}{{\text{\bf mp}}} 

\begin{document} 
\font\eightrm=cmr8  \font\eightit=cmsl8 \let\it=\sl  

\title[Relatively nilpotent extensions]{Moduli of relatively nilpotent algebraic 
extensions}

\author[M.~Fried]{Michael D.~Fried} 
 
\date{\today} 
\newcommand\rk{{\text{\rm rk}}} 
\newcommand{\TG}[2]{{{}_{#1}^{#2}\tilde{G}}} 
\newcommand{\vsmatrix}[4]{{\left(\smallmatrix #1 & #2 \\ #3 & #4  
\\ \endsmallmatrix \right)}} 
\newcommand{\tbg}[1]{{{}^{#1}\tilde\bg}} 
\newcommand{\Inn}{{\text{\rm Inn}}} 
\newcommand{\af}{{\text{\rm af}}} 
\newcommand{\f}{{\text{\rm f}}} 
 
\address{UC Irvine, Irvine, CA 92697, USA} 
\email{mfried\@math.uci.edu}  

\begin{abstract} The  Main Conjecture on Modular Towers: For a prime
$p$, and finite $p$-perfect group $G$, high levels of a $(G,p)$-Modular Tower 
have no rational
points. When $p$ is an odd prime and $G$ is the dihedral
group $D_p$ we call the towers  Hyper-modular.
\cite{BFr} proves cases of the Main Conjecture for Modular Towers with $p=2$ and 
$G$ an
alternating group. We use differences between Hyper-modular
and general Modular Towers to give new moduli applications. Their
similarities give these applications insights successful with modular curves. 

The universal $p$-Frattini cover of $G$ and a collection of $p'$
conjugacy classes define a
particular Modular Tower's levels. The number of components at a level and how 
cusp ramification
grows from level to level relate to the appearance of Schur multipliers. 
\cite{BFr} 
applies formulas of the author and Serre to spin covers to handle the case 
$p=2$.
We give here an  exposition on going beyond that analysis. 

As with modular curves, the tower levels are moduli spaces.   
We use a Demjanenko-Manin result and a construction from \cite{FrMT}  to 
more directly approach rational points than does 
Falting's Theorem.

\end{abstract}


\thanks{Conversations with Darren Semmen were extremely valuable in reporting on 
progress
in this exposition extending my talks at RIMS  in October 2001. Thanks to NSF 
Grant
\#DMS-9970676 and
\#DMS-0202259 and RIMS for support. Special thanks to Yasutaka Ihara and his 
four colleagues 
Makoto Matsumoto, Shin Mochizuki, Hiroaki Nakamura and Akio Tamagawa for an 
apposite and challenging
conference.} 

\maketitle 
\def\addcontentsline#1#2#3{%
\addtocontents{#1}{\protect\contentsline{#2}{#3}{}}}

\setcounter{tocdepth}{2}
\tableofcontents

\section{Basic goals of Modular Towers} Suppose the prime $p$ divides the order 
of  a finite group $G$. We
say $G$ is 
$p$-perfect if $G$ has no $\bZ/p$ quotient. A {\sl perfect\/} group is $p$-
perfect for
each prime dividing its order. Equivalently,  $G$ equals its commutator
subgroup. A $p'$ conjugacy class (subgroup, element, etc.) of $G$ has 
elements of order prime to $p$. Being $p$-perfect is equivalent to having $p'$
generators. 

We think of {\sl Hyper-modular Towers\/} relating to {\sl Modular Towers\/} as 
$(D_p, p)$ ($p$ odd) relates to $(G,p)$, with $G$ a $p$-perfect finite group. A
Hyper-modular Tower comes with an even number 
$r$ of involution conjugacy classes. A 
modular curve tower has $r=4$ and its $k$th level is the space $X_1(p^{k+1})$ 
without its cusps.  A  
$(G,p)$ Tower has a collection of $p'$ conjugacy classes $\bfC=(\row \C r)$.   
We call $r$
the {\sl p-dimension\/} of a Nielsen class. Then $r-3$ is its {\sl j-dimension}. 
The  p-dimension 
is the dimension of the Hurwitz space tower levels. Each level is an
affine cover of $\prP^r\setminus D_r$, projective $r$ space minus its 
discriminant locus.  

Let $\prP^1_z=\bC_z\cup
\{\infty\}$ be projective 1-space with a uniformizing variable $z$.
Recall the space $J_r$: $\PGL_2(\bC)$ equivalence classes of unordered distinct 
points $\prP^1_z$. Then,
the {\sl reduced\/} Hurwitz space at the $k$th level of a
$(G,p,\bfC)$ modular tower has dimension equal to the j-dimension. We explain 
below the notation 
$\sH(G_k,\bfC)^{\inn,\rd}=\sH(G_k,\bfC)^\rd$ ($G_0=G$) for this inner
Hurwitz space. It has an affine (so surjective with finite fibers) map to 
$J_{r}$. 

When $r=4$ the Modular Tower levels are covers of 
$\prP^1_j\setminus\{\infty\}=J_4$,
the classical $j$-line, and they are quotients of the upper half-plane by a 
finite index subgroup of
$\PSL_2(\bZ)$. We normalize so their ramified points are $j=0, 1, \infty$ with 
any ramified points over 0
(resp.~1) having index 3 (resp.~2). The points of the projective completion of 
the curve levels  
over
$\infty$ are {\sl cusps\/}. Computing their {\sl widths\/} (ramification 
indices) is subtle and significant
(\S\ref{ramindex}). 

\subsection{Topics of this paper and the Main Conjecture} An hypothesis on $G_0$ 
seeding a Modular Tower is
always in force, unless otherwise said, throughout this paper: That $G_0$ is a 
centerless $p$-perfect group. We
often remind the reader of that by using the phrase $p$-perfect, though it also 
means centerless so as to apply
the  consequence from Thm.~\ref{pdivis} that all the $G_k\,$s also are 
centerless. The Main Conjecture is that
for
$(G,p,\bfC)$ above (with $G=G_0$ a $p$-perfect centerless group) and
$k$ large,  
$\sH(G_k,\bfC,p)^{\inn,\rd}$ has no rational points. We state a geometric 
version. 

\begin{guess}[Geometric Conjecture] \label{gconj}  For $k$ large, nonsingular 
projective completions
of all components of a Modular Tower at level $k$ have general type.  
\end{guess} 

Characteristic
$p$-Frattini covers $\{G_k\}_{k=0}^\infty$ of $G=G_0$  define the
reduced Modular Tower levels, $\sH(G_k, \bfC)^{\inn,\rd}=\sH_k^\rd$. Since 
$\bfC$ consists of $p'$
conjugacy classes, each $\C_i$ pulls back to unique conjugacy classes in each
covering group $G_k\to G_0$. That means the notation for
$(G_k,\bfC)$  reduced Nielsen classes makes sense (\S\ref{shincidence}). 

When $r=4$, the genus of a level $k$ component depends on elliptic
ramification and cusp ramification, respectively,
$\ind(\gamma_0)+\ind(\gamma_1)$ and  $\ind(\gamma_\infty)$  in
\eql{compgenus}{compgenusc}. For Modular Towers these engage us in a moduli
interpretation related to $(G_k,p)$. Though this is akin to the
simplest case $(D_{p^{k+1}},p)$ for modular curves, there are four new
challenges. 

Two deal with the moduli interpretation of $\gamma_0$
and $\gamma_1$ fixed points and for {\sl orbit shortening\/}. Our goal is to 
find a
level at which each disappears (\S\ref{ellipram}). The 3rd is to locate those
$p$-divisible cusps at level $k+1$ that lie above level $k$ cusps 
that aren't $p$-divisible. These are the contributors to the
$U_i\,$s of \eqref{genbound}. We nailed these contributions in the
\cite{BFr} examples for $p=2$ as coming from spin covers of finite groups. This 
used formulas of the
author and Serre. A harder  spin cover analysis then found the
components of these example Modular Tower levels.

Here is a list of  topics in this paper. Several of these include the case 
$r\ge 5$.  

\begin{edesc} \label{topics} \item \label{topics3}  \S\ref{shap}: Computing 
effectively the characteristic
$p$-Frattini module
$M_k$ inductively defining $G_{k+1}$ from $G_k$ (hinted at in 
\cite[Rem.~2.10]{FrMT}). 

\item \label{topics1} \S\ref{MainConjTruth}:  Heuristics for the Main
Conjecture on  Modular Tower levels when 
the j-dimension is 1. 
\cite[\S 8.1]{BFr}. 
\item \label{topics4} \S\ref{typeschur}: Classifying for all $p$ how 
 Schur multipliers figure in a change of $p$-divisible cusps from level $k$ to 
level $k+1$.  
\item \label{topics2} \S \ref{mandem}: The Demjanenko-Manin effective 
diophantine approach  
when the j-dimension is 1, and its potential for all Modular Towers. 
\item \label{topics5} \S\ref{projsystems}: Situations assuring there are 
projective systems of
absolutely irreducible $\bQ$ components on each Modular Tower level.  
\end{edesc}

\subsection{Reduced Hurwitz space components} \label{ramindex}
Each Modular Tower comes with 
$p'$ conjugacy classes $\bfC=(\row {\C} r)$. These give the Nielsen class 
elements:
$$\ni=\ni(G,\bfC)=\{\bg= (\row g r) \mid g_1\cdots g_r=1, 
\lrang{\bg}=G \text{\ 
and\ } \bg\in \bfC\}.$$ The notation $\bg\in\bfC$ means there is a $\pi\in S_r$ 
with $g_{(i)\pi}\in
\C_i$, $i=1,\dots,r$. The elements of {\sl inner\/} equivalence classes are the 
collections
$\{h^{-1}\bg h\}_{h\in G}$. Suppose $r$ is even and $G$ has $r/2$ conjugacy 
classes $\row {\C'} {r/2}$ so 
with elements $g_i\in \C'_i$, $i=1,\dots,r/2$, that generate $G$. Let $\bfC$  be 
the
conjugacy classes with $\C_{2i}=\C_i'$, $\C_{2i-1}=(\C_i')^{-1}$, 
$i=1,\dots,r/2$. Then, $\ni(G,\bfC)$
contains this Harbater-Mumford representative (H-M rep.): $(g_1^{-1},g_1,g_2^{-
1},g_2,\dots,
g_{r/2}^{-1},g_{r/2})$. The relation to cusps of H-M reps.~appears repeatedly in 
\cite{BFr}.

Many applications use Hurwitz spaces with points that are equivalence classes of
$(G,\bfC)$ covers having branch cycle descriptions from a permutation 
representation $T:G\to S_n$. Then
the covers have degree
$n$ and the corresponding equivalence classes $\{h^{-1}\bg h\}_{h\in 
N_{S_n}(G)}$ with $N_{S_n}(G)$ the
normalizer of $G$ in $S_n$. Example: The case with $G=D_p$ ($p$ odd, $r=4$) 
above produces the
modular curve
$Y_0(p)=X_0(p)\setminus \{\text{cusps}\}$ using 
the standard degree
$p$ representation, $N_{S_p}(D_p)=\bZ/p\xs (\bZ/p)^*$ and reduced equivalence 
(below) of covers. When,
however,  we equivalence only by conjugation by
$D_p$ (inner classes), the result is $Y_1(p)$. Permutation representations 
appear in any precise
analysis of Modular Towers. The Main Conjecture, however, assumes reduced inner 
equivalence on Nielsen
classes.We concentrate here on that case; conjugation by $G$ gives Nielsen 
classes. 

When 
$r$ is at least 4, a braiding action gives important invariants of 
 the reduced Hurwitz space
$\sH(G,\bfC)^\rd$. 
Each component corresponds to an
orbit for the action of two operators:  
\begin{edesc} \item The shift: $(\row g r)\mapsto (g_2,\dots,g_r,g_1)=(\bg)\sh$; 
and 
\item The 2-twist: $(\row g r)\mapsto (g_1,g_2g_3g_2^{-
1},g_2,\dots,g_r)=(\bg)\gamma_\infty$. 
\end{edesc}
You may replace the 2-twist by the $i$-twist $q_i$, $1\le i\le {r-1}$. Just one 
twist and
the shift generate all the braidings. The cases $r=4$ and $r\ge 5$ differ 
slightly. 

\subsection{The $\sh$-incidence matrix for $r=4$} \label{shincidence} 
\cite[Prop.~4.4]{BFr} shows how to
compute the genuses of $\sH(G,\bfC)^\rd$ components.  This uses a Klein 4-group
$\sQ''$ that acts on Nielsen classes. We give its generators in $\row q {r-1}$
notation:
\begin{edesc} \item $q_1q_3^{-1}$: $(g_1,\dots,g_4)\mapsto (g_1g_2g^{-
1}_1,g_1,g_3g_4g_3^{-1},g_3)$. 
\item $\sh^2$: $(g_1,\dots,g_4)\mapsto (g_3,g_4,g_1,g_2)$. 
\end{edesc}

Form {\sl
reduced\/} (inner) Nielsen classes:
$G\backslash
\ni(G,\bfC)/\sQ''=\ni(G,\bfC)^\rd$. With $\gamma_1=\sh$ and $\gamma_\infty$ the 
2-twist (or middle
twist) acting on $\ni(G,\bfC)^\rd$, we draw conclusions for the  $\bar
M_4=\lrang{\gamma_1,\gamma_\infty}$ action. 
\begin{edesc} \label{compgenus} \item \label{compgenusa} Components of 
$\sH(G_k,\bfC)^\rd$ correspond to
$\bar M_4$ orbits $\bar O$ on
$\ni(G_k,\bfC)^\rd$. Use $\sH_{k,\bar O}$ for the level $k$ component 
corresponding to $\bar O$.  
\item \label{compgenusb} The cusps of $\sH_{k,\bar O}$ correspond to 
$\gamma_\infty$ orbits $\row O t$ on
$\bar O$;  cusp widths are the orbit lengths.  
\item \label{compgenusc} The genus $g_{\bar O}$ of $\sH_{k,\bar O}$ appears in 
the formula 
$$2(|\bar O|+g_{\bar O}-1)=\ind(\gamma_0) + \ind(\gamma_1) + 
\ind(\gamma_\infty), \text{ with }
\gamma_0=(\gamma_1\gamma_\infty)^{-1}.$$
\end{edesc} 

List as $\row O u$ all $\gamma_\infty$ orbits on $\ni(G,\bfC)^\rd$, without 
regard in which $\bar M_4$
orbit they fall. The symbol $(O_i)\sh$ means to apply $\sh$ to each reduced 
equivalence class in
$O_i$. The $(i,j)$ term of the 
$\sh$-incidence matrix
$A(G,\bfC)$ is  
$|(O_i)\sh\cap O_j|$. As $\sh$ has order two on reduced Nielsen classes, this is 
a symmetric
matrix. Reorder $\row O u$ to arrange $A(G,\bfC)$ into blocks along the 
diagonal. Each
block corresponds to an irreducible component of $\sH(G,\bfC)^\rd$ 
\cite[Lem.~2.26]{BFr}.   
\cite[\S 2.10]{BFr} shows how efficient is the shift incidence matrix in 
computing orbits by doing the  case
$(A_5,\bfC_{3^4})$ with $\bfC_{3^4}$ indicating four repetitions of the 
conjugacy class of 3-cycles. The
harder case $(G_1,\bfC_{3^4})$ appears in \cite[\S 8.5]{BFr}. 
The $\sh$-incidence matrix works for general $r$ similarly, though for  $r\ge 5$ 
there is 
no corresponding group $\sQ''$ and the 
matrix is no longer symmetric
\cite[\S2.10.2]{BFr}.

\section{$p$-perfect groups and associated 
modular curve-like towers} Let $G$ be a finite group with $p$  a prime  dividing 
$|G|$. Suppose
$r\ge 3$ and we have an $r$-tuple of $p'$ conjugacy classes of
$G=G_0$. This produces a natural tower $\{\sH_k\}_{k=0}^\infty$ of affine
moduli spaces covering $\prP^r\setminus D_r$. If $G$ is $p$-perfect, then 
Thm.~\ref{pdivis} says each tower level is a fine moduli space. 

\subsection{Module theory for the split case} \label{psplit} With $\bZ_p$ the 
$p$-adic
integers, let  $-1$ acting on $\bZ_p$ by multiplication. For $p$ odd, $\bZ_p\xs 
\{\pm1\}\to \bZ/p\xs
\{\pm 1\}$ is a cover of $p$-perfect groups. With $p$ dividing the order of $G$, 
we generalize this by
considering the universal
$p$-Frattini  cover $\phi: \tG p\to G_0$. Then,  $\phi$ has a pro-free pro-$p$
kernel. In the  
$p$-split case, $G=G_0=P_0\xs H$ with $P_0$ the $p$-Sylow of $G$ and $H$ is a 
$p'$ group. The 
$H$ action extends to the minimal pro-free pro-$p$ cover $\phi: \tilde P\to P_0$
\cite[Prop.~5.3]{BFr}. We  explain the characteristic groups in this easy case.  

We inductively consider $\ker_0=\ker(\phi)$ and $\ker_k$,  the $k$th iterate of 
the Frattini subgroup of $\ker_0$. So, $\ker_1$ is the Frattini 
subgroup of $\ker_0$, generated by $p$th powers and commutators of 
$\ker_0$; $\ker_2$ is he Frattini subgroup of $\ker_1$, etc. The 
action of $H$ extends to $\tilde P$ and to $\tilde P/\ker_k =P_k$, 
$k\ge 0$, in many ways \cite[Rem.~5.2]{BFr}. Denote the
characteristic quotients of
$\tilde P\xs H$ by  
$\{G_k=G_k(P_0\xs H)=P_k\xs H\}_{k=0}^\infty$. Define $M_k$ as 
$\ker_k/\ker_{k+1}$. This is a right $G_k$ module as follows: 
$g\in G_k$ maps $k\in \ker_k/\ker_{k+1}$  to $\hat 
g^{-1} k\hat g$ (right action), with $\hat g$ any lift of $g\to 
G_{k+1}$.

\begin{lem}  Let $U$ be the Frattini subgroup of $P_0$,
and let $V^*$ be the dual space to $P_0/U=V$. We characterize $G_0=P_0\xs H$
being $p$-perfect by the induced map of $H$ on $V^*$ has no
nontrivial $H$ invariant vector $v^*\in V^*$. \end{lem}

\begin{proof} Suppose   $v^*\in V^*\setminus \{0\}$  is $H$ invariant. Let $U'$ 
be the
pullback to $P$ of  $\ker(v^*)$. Consider $\psi: P_0\xs H\to \bZ/p$ by 
$hk\mapsto
k \mod U'$, for $h\in H$ and
$k\in P_0$. We check this is a homomorphism: $hkh'k'=hh'((h')^{-1}kh')k'$
and $(h')^{-1}kh' \mod U' \equiv k \mod U'$. Reverse the steps for the converse.
\end{proof} 

Let $K$ be a field. For any finite group $G$, the group ring 
$\Lambda=K[G]$ is an augmentation algebra from the ring homomorphism  
$\aug: \sum_i a_ig_i \mapsto \sum_i a_i$. The quotient vector space 
from $\aug$ is the identity module $\one_{G}=K$ for $\Lambda$ from 
$$\sum_i a_ig_i\alpha=(\sum_ia_i)\alpha \text{ for }\alpha\in 
\Lambda/\ker(\aug).$$ Also, maximal 2-sided ideals are in 
$\ker(\aug)$. So, the Jacobson radical $\rad(\bF_p[G])$, the 
intersection of he maximal 2-sided ideals of $\Lambda$, is also. When 
$G$ is a $p$-group, the Jacobson radical is exactly $\ker(\aug)$. 

\begin{defn}[Involution pairing]  Suppose 
$V = K^n$, and $h\in S_n$ has order 2 acting as permutations on the standard
basis vectors for $V$. Define an inner (dot) product pairing: 
$v_1=(x_1,...,x_n)$ and $v_2=(y_1,...,y_n)$ 
pair to $\lrang{v_1,v_2}_h = \sum_{i=1}^n x_i y_{(i)h}$. Then,  $\lrang{\cdot 
,\cdot }_h$ is symmetric.
\end{defn}

\begin{defn} Let $\Lambda$ be a commutative, associative algebra with unit
$1_\Lambda$. Suppose there is a nondegenerate inner product $\lrang{,}$ with
$\lrang{a b,c} = \lrang{a,b c}$, $a,b,c\in \Lambda$. We call $\Lambda $ a {\sl 
Frobenius\/}  algebra.
\end{defn}

\begin{exmp} With $K$ a field, let $\Lambda=K[G]$.  Elements of $G$ form a $K$ 
basis of $\Lambda$. 
We take
$h=h_\inv$ to permute $G$ by  $g \mapsto g^{-1}$. 
Then, $\lrang{a,b}$ is the coefficient of $1_G$ in $ab$, and $\lrang{a b,c}$  is 
the coefficient of $1_G$
in $abc$. So, associativity of multiplication in $K[G]$ makes $K[G]$  a 
Frobenius algebra.
\end{exmp}

For $\Lambda$ a symmetric Frobenius algebra, any left $\Lambda$ module is 
isomorphic to
its dual $\Lambda^* =\Hom_k(\Lambda,k)$, a right $\Lambda$ module.  Conclude: 
The 
projective and injective $\Lambda$ modules are the same. A {\sl primitive
idempotent\/} of $\Lambda$ is one we can't decompose as a nontrivial sum of
orthogonal idempotents. These correspond to projective
indecomposables. 

Let $M$ and $N$ be right $K[G]$ modules. The {\sl Hopf algebra\/} structure on 
$K[G]$ comes from $\sum_i
a_ig_i
\mapsto 
\sum_i a_ig_i\otimes g_i$. This gives $M\ot_{K} N$ a
$K[G]$ module structure. The projective indecomposables $P$
satisfy $\soc(P)\equiv P/\rad(P)$. So, the simple modules $G$ at
the top and bottom {\sl Loewy layers\/} (see \S\ref{radlayers}) of a projective 
indecomposable module are the
same. 

\subsection{Module theory for $(G_0,p)$ in the general case} Let $\tilde P$ be a 
pro-free pro-$p$
group with the same rank as the $p$ group $P_0$. Assume $G_0'=P_0\xs H$ is the 
$p$-split case. Then, 
$\tilde P\xs H$ is the universal $p$-Frattini cover of $G_0'$. Now we do the
nonsplit case. Every group
$G=G_0$ with
$p$ dividing 
$|G_0|$ has a (nontrivial) universal $p$-Frattini cover ${}_p\tilde {G_0}$. It 
is versal for covers of $G_0$
with
$p$-group kernel as is 
$\bZ_p\xs \{\pm 1\}$ versal for covers of $D_p$ with $p$-group kernel 
(\cite[Chap.~21]{FrJ}, \cite[Part
II]{FrMT} and \cite[\S3.3]{BFr}). 

It is harder, however, to describe ${}_p\tilde G_0$ for general $G_0$. Let $P_0$ 
be a $p$-Sylow
of
$G_0$. Use  $G_0'$ for the normalizer in $G_0$ of $P_0$. (Don't confuse $G_0'$ 
with the commutator of
$G_0$.) By Schur-Zassenhaus, 
$G_0'=P_0\xs H$ with $H$ a (maximal) $p'$-split quotient of $G_0'$. Let  
$\phi_{k, t} :
G_k\to G_t$, $t\le k$ be the natural map. We apply $'$ to modules and groups 
associated with the
universal
$p$-Frattini cover of $G_0'$ in
\S\ref{psplit}:
$$\{M_k'=\ker(G_{k+1}'\to G_k')=\ker(\phi_{k+1,k}')\}_{k=0}^\infty.$$ 

\subsubsection{Inducing from $G_0'$ to $G_0$}
Now remove the $'$, as in such notation as $\{M_k= 
\ker(\phi_{k+1,k})\}_{k=0}^\infty$, for groups and 
modules corresponding to $G_0$. We regard $\bF_p[G_0]$ as a left $\bF_p[G_0']$ 
module and as a right $\bF_p[G_0]$ module. The induced module 
$\ind_{G_0'}^{G_0}(M_0')$ is the $G_0$ module 
$$M_0'\ot_{\bF_p[G_0']}\bF_p[G_0]=M_0'\ot_{\bF_p} \bF_p[G_0/G_0'].$$ The 
notation $\bF_p[G_0/G_0']$ is for the right $G_0$ module written as the vector 
space generated by  
right cosets of $G_0'$ in $G_0$. Then,  
$\ind_{G_0'}^{G_0}(M_0')$ is a right $G_0$ module. 

Suppose $N$ is a right $G_0$ module.  Any $\bZ/p[G_0']$ homomorphism  $\psi: 
M\to N$ extends to a
$\bZ/p[G_0]$ module homomorphism $\Ind_{G_0'}^{G_0}(M)\to M$ by $m\ot g \mapsto 
\psi(m)^g$. Recall that
$M_0$ is an indecomposable $G_0$ module (\cite[p.~11,
Exec.~1]{Ben1} or 
\cite[Indecom.~Lem.~2.4]{FrKMTIG}). 
To characterize $M_0$ as the versal module for exponent $p$
extensions of $G_0$, we use this result \cite[Prop.~2.7]{FrMT}. 

\begin{prop} \label{1-dim} The cohomology group $H_2(G_0,M_0)$ has dimension one 
over $\bF_p$. The 2-cocycle for
the short exact sequence $$1\to M_0\to G_1\to G_0\to 1$$ represents a generator.  
We define any nontrivial 
$\alpha\in H_2(G_0,M_0)$ as 
$G_1$ up to an automorphism fixed on the
$G_0$ quotient and multiplying $M_0$ by a scalar. \end{prop} 

\begin{rem}[Automorphisms of $M_0$] Suppose $G$  acts on a vector space $M$. 
Assume $L:M\to M$ is a
linear map (though our notation is multiplicative) commuting with $G$. Let   
$\psi\in H^2(G,M)$. So,
$\psi$ makes $G\times M$ into a group through this formula:
$$ (g_1,m_1)*(g_2,m_2)=(g_1g_2,m_1^{g_1}m_2\psi(g_1,g_2)),\ g_1,g_2\in G,\ 
m_1,m_2\in M.$$ By replacing
$\psi(g_1,g_2)$ by $L(\psi(g_1,g_2))\eqdef\psi_L(g_1,g_2)$ we preserve the 
cocycle condition.
Scalar multiplication by $u\in \bZ/p$  replaces $\psi(g_1,g_2)$ by 
$\psi(g_1,g_2)^u$. Even
when  $M_0$ is the characteristic $p$-Frattini module for $G_0$, there may be 
nontrivial  maps
$L$ (not given by scalar multiplication) that preserve the cocycle. This 
produces nontrivial
automorphisms of
$G_1$ that induce the identity on
$G_0$. Such an $L$ exists when $G_0=K_4$, the Klein 4-group, and $M_0$ has 
dimension 5 in \S \ref{VDs}. 
\end{rem} 

\subsubsection{Using the indecomposability of $M_k$ as a $\bZ/p[G_k]$ module} 
\label{shap} 
We use an observation on the
natural extension
$\phi_{1,0}: G_1\to G_0$.  

\begin{lem} \label{nonsplitp} The extension $\phi_{1,0}^{-1}(G_0')\to G_0'$ is 
not split. 
\end{lem}

\begin{proof} We apply Thm.~\ref{pdivis}. Then, any $g\in G_0$ having order $p$  
lifts to have order $p^2$ in
$G_1$. This holds if $g$ is in a $p$-Sylow $P_0$ of $G_0'$. This is a 
contradiction if  
if $\phi_{1,0}^{-1}(G_0')\to G_0'$ splits.  
\end{proof} 

\begin{prop} \label{propnonsplit} A natural extension 
$\Ind_{G_0'}^{G_0}(M_0')\to G^*
\mapright{{\alpha^*}} G_0$ induces a surjective $G_0$ module homomorphism 
$\psi: \Ind_{G_0'}^{G_0}(M_0')\to M_0$. This gives a
$G_0$ splitting
$M\oplus N$ of $\Ind_{G_0'}^{G_0}(M_0')$ that makes $G^*/N$ isomorphic to $ 
G_1$. So,
if 
$\Ind_{G_0'}^{G_0}(M_0')$ is indecomposable,  we see that $G^*$ is isomorphic to 
$G_1$ as a cover of $G_0$. 
\end{prop}

\begin{proof}  An element $\alpha'\in H^2(G_0',M_0')$ defines the extension 
$G_1'\to G_0'$. By Shapiro's
Lemma, $H^2(G_0',M_0')=H^2(G_0,
\Ind_{G_0'}^{G_0}(M_0'))$ \cite[p.~42]{Ben1}. So $\alpha'$ canonically defines  
$\alpha\in H^2(G_0,
\Ind_{G_0'}^{G_0}(M_0'))$. The extension
$\Ind_{G_0'}^{G_0}(M_0')\to G^*
\mapright{{\alpha^*}} G_0$ represents $\alpha$. Regard 
$\{\alpha(h,h')\}_{h,h'\in G_0'}$ as giving  a
multiplication on $G_0'\times M_0'$: $(h,m)*(h',m')=(hh',m^{h'}m'\alpha(h,h'))$. 
The associative law for
this multiplication is the cocycle condition. Extend it to 
$\Ind_{G_0'}^{G_0}(M_0')\times G_0$ using the
following rules involving right coset representatives $\row g n$ with 
$n=(G_0:G_0')$:   
\begin{edesc} \item $m^{g_i}=m\ot g_i$, $m\in M_0'$; and 
\item $\alpha(hg_i,h'g_{i'})=\alpha(h,h'')$ with $h''$ satisfying the equation 
$h''g_j=g_ih'g_i^{-1}$ for
some $j$. \end{edesc}

Since $G^*\to G_0$ is a cover with exponent $p$ kernel, there exists
$\beta: G_1\to G^*$ covering $G_0$. Let
$H^*=(\alpha^*)^{-1}(G_0')$ and
$H_1=\phi_{1,0}^{-1}(G_0')$. 
There exists a homomorphism $\gamma: G_1'\to H_1$ because
$G_1'$ is the universal exponent
$p$ Frattini cover of $G_0'$. Lemma \ref{nonsplitp} says $H_1\to G_0'$ does not 
split.
Further, the induced map $M_0'\to M_0$ extends to a map of
$G_0$ modules $\Ind_{G_0'}^{G_0}(M_0') \to M_0$ by $ m_i'\ot g_i \mapsto 
\gamma(m_i)^{g_i}$. 

So, the cocycle defining the
extension $G^*\to G_0$ maps to a cocycle defining an extension $\psi: 
G^\dagger\to G_0$ with
$\ker(\psi)=M_0$. Since the restriction of $\psi$ over $G_0'$ is nonsplit, 
Prop.~\ref{1-dim} says $\psi:
G^\dagger\to G_0$ identifies with the extension $G_1\to G_0$. This produces the   
corresponding homomorphism
of groups
$G^*\to G_1$. We also call this 
$\gamma$. As
$G_1\to G_0$ is a Frattini cover, $\gamma$ must be surjective. 

Consider $\beta\circ \gamma=\mu$ and let $\mu^{(t)}$ be its $t$th
iterate restricted to $\Ind_{G_0'}^{G_0}(M_0'))$. We apply Fitting's Lemma 
\cite[Lem.~1.4.4]{Ben1}. Conclude,  
for suitably high $t$,  
$\Ind_{G_0'}^{G_0}(M_0')$ decomposes as a $G_0$ module into a direct sum of the 
kernel and range of
$\mu^{(t)}$. The range  will be an indecomposable $G_0$ summand isomorphic to 
the
indecomposable module
$M_0$. If 
$\Ind_{G_0'}^{G_0}(M_0')$ is already indecomposable, it must be isomorphic to 
$M_0$. 
\end{proof}

\subsubsection{A difference between $M_0$ and $M_k$, $k\ge 1$} 
\label{changeLevel1} For $k\ge 1$,
the universal $p$-Frattini cover of $G_k$ is still $\tG p$. We see this from the 
characterization of $\tG p$
as the minimal $p$-projective cover of $G_k$ \cite[Prop.~20.33 and 
Prop.~20.47]{FrJ}. By its definition, $M_k$
is a $G_k$ module. So, for $k\ge 1$, it is also the 1st characteristic module 
for the $p$-Sylow of $G_k$. 

We designed Prop.~\ref{propnonsplit} to handle level 0 where there can be a
distinctive difference between $M_0(P_0)$ and $M_0(G_0)$. Ex.~\ref{A55} shows 
the most extreme case of this,
according to Prop.~\ref{propnonsplit}, with $\Ind_{G_0'}^{G_0}(M_0')=M_0$.  
There are several types of
intermediate situations by considering  groups
$G^\dag$, properly containing $N_{G_0}(G_0')$ with $M_0(P_0)$ a $G^\dag$ module 
extending the action of
$N_{P_0}(G_0)$. To assure $M_0(G^\dag)=M_0'$, we need a nontrivial $\alpha\in
H^2(G^\dag,M_0(P_0)$.  Since $(G^\dag:N_{P_0}(G_0)),p)=1$, restriction to the 
one-dimensional
$H^2(N_{P_0}(G_0),M_0(P_0))$ (Prop.~\ref{1-dim}) is injective. So, the extension 
for $\alpha$ gives
the group $G_1$ for $G^\dag$. 

An example of this is $G_0'=A_4\le G_0=A_5$ and $p=2$. This shows the other
extreme of Prop.~\ref{propnonsplit}: $M_0'=M_0(P_0)=M_0(G_0)$ 
(\cite[Prop.~5.4]{BFr} for a quick constructive
proof, or \cite[Prop.~2.9]{FrMT}). 

\begin{rem}[Extending Prop.~\ref{propnonsplit}] \label{extnonsplit} Suppose 
the situation above occurs with a nontrivial $\alpha$ for $G^\dag$. The argument 
for Prop.~\ref{propnonsplit}
applies by replacing $\Ind_{G_0'}^{G_0}(M_0')$ with $\Ind_{G^\dag}^{G_0}(M_0')$. 
We could 
have  a complicated sequence applying this process before confidently 
proclaiming the exact summand of
$\Ind_{G_0'}^{G_0}(M_0')$ giving $M_0(G_0)$. 

\end{rem}  

\subsection{Radical layers and the appearance of $\one_{G_k}$} \label{radlayers}
{\sl Radical layers\/} of $M_k$ are the $G_k$ module quotients
$\rad(\bZ/p[G_k])^jM_k/\rad(\bZ/p[G_k])^{j+1}M_k$, 
$j\ge 0$.  These radical layers gives a maximal sequence
of
$G_0$ submodule quotients, with each a direct sum of simple
$G_0$ modules. The  Loewy display captures this data. It includes arrows showing 
how 
simple modules from different layers form subquotient extensions. One example 
occurs
in Ex.~\ref{A55}.  
\cite[Cor.~5.7]{BFr} exploits one case for $A_5$ and $p=2$. We see it in its 
restriction to
$A_4$ and to $K_4\le A_4$ in \S\ref{VDs}.  

Appearances of $\one_{G_k}$, the trivial $G_k$ module, in the Loewy layers of 
$M_k$ usually effect 
the structure of the $k$th level of a Modular Tower. The appearance of
$\one_{G_k}$ at the  tail of $M_k$ interprets that $G_{k+1}$ has a center. 
Thm.~\ref{pdivis} 
includes one reason for the hypothesis that $G_0$ is $p$-perfect 
\cite[Prop.~3.21]{BFr}.

\begin{thm} \label{pdivis} Suppose $p$ divides the order of  $g\in G_k$.  Then, 
any lift
$\hat g
\in G_{k+1}$ has  order $p\cdot \ord(g)$. Assume $g\in G_k$ is a $p'$ element. A  
unique $p'$ conjugacy class
of $G_{k+1}$ lifts $g$. If $G_0$ is centerless and p-perfect, so is $G_k$ for 
all
$k$. \end{thm}

\begin{exmp}[$A_5$ and $p= 5$] \label{A55}  The normalizer of
$P_0$ in
$A_5$ is a dihedral group $G_0'$ and $M_0'$ is $\bZ/5$. From 
Prop.~\ref{propnonsplit}, 
$M_0$ is  an indecomposable component of the rank $(G_0:G_0')=6$ module 
$\ind_{G_0'}^{G_0}M_0'$. There is  an
obvious 5-Frattini cover $$\phi': \PSL_2(\bZ/5^2) \to \PSL_2(\bZ/5)=A_5.$$ The 
kernel of $\phi'$ is the adjoint
representation $U_3$ for $\PSL_2(\bZ/5)$. The rank of $M_0$ determines the rank 
(minimal number of generators)
of
$\ker_0$ as a pro-free pro-5 group. Conclude, if the rank of
$M_0$ is 3, then $\phi:\PSL_2(\bZ_5)\to\PSL_2(\bZ/5)$ would be the universal 5-
Frattini cover of $A_5$.
The kernel, however, of $\phi$ is not a pro-free group. Since there is no rank 2 
simple
$\bZ/5[A_5]$ module, the Loewy display for $M_0$ either has three copies of the 
trivial representation
in it, it is $U_3\oplus U_3$, or it is $U_3 \to U_3$. The first fails 
Thm.~\ref{pdivis} for then 
$G_1$ would have a nontrivial center. The second fails indecomposability of 
$M_0$. So, the last of these
gives $M_0$. 
\end{exmp} 

\section{Heuristics for the Main Conjecture when $r=4$} \label{MainConjTruth} 
Assume $r=4$, $G=G_0$ is $p$-perfect and $\bfC$ is a collection of
$p'$ conjugacy classes of $G_0$ for which $\ni(G_k,\bfC)$ is nonempty
for each integer $k\ge 0$. We given an intuitive justification why each 
component
of $\sH(G_k,\bfC)^\rd$ has large genus if $k$ is large.

\subsection{The setup} \cite[\S 9.6]{BFr} assumes we have no 
control over level 0 component genuses. So, a test for genus growth starts  
with assuming level 0 has a genus 0 component. Let
$\bar O_k$ be a
$\bar M_4$ orbit on $\ni(G_k,\bfC)^\rd$  at level
$k\ge 0$.  Assume its corresponding component has genus 0.  Suppose $\bar 
O_{k+1}$ is a $\bar M_4$ orbit on
$\ni(G_{k+1},\bfC)^\rd$ lying above it.  Let $\row O t$ be the $\gamma_\infty$ 
orbits on $\bar O_k$.

If a level $k$ component has genus 1, then it is usually easy to find evidence 
of some
ramification from $\bar O_k$ to $\bar O_{k+1}$. This assures the genus of $\bar 
O_{k+1}$ is at least 2, and the
genus rises for orbits at higher levels over $\bar O_{k+1}$  automatically. So,
we give simple reasons why we expect the rise of the genus follows from 
\eqref{compgenus}. Our setup is to compute the genus of a component 
(corresponding to a $\bar M_4$ orbit) $\bar
O_{k+1}$ at level $k+1$ lying over a component $\bar O_k$ at level $k$, assuming 
$\bar O_k$ has
genus 0.  

\subsection{Elliptic ramification and orbit shortening} \label{ellipram} We 
denote the component corresponding
to
$\bar O_{k+1}$ by
$\sH^\rd_{\bar O_{k+1}}$. Each 
$\bp_{k+1}\in \sH_{\bar O_{k+1}}$ represents a cover $$\phi_{\bp_{k+1}}: 
X_{\bp_{k+1}} \to \prP^1_z.$$
Suppose
$\bp_{k+1}$ over 0 or 1 on $\prP^1_j$ actually ramifies over its  image
$\bp_k\in
\sH^\rd_{\bar O_{k+1}}$. So, it contributes 2 (resp.~1) to  $\ind(\gamma_0)$
(resp.~$\ind(\gamma_1)$) on the right of \eql{compgenus}{compgenusc} (when $\bar 
O=\bar O_{k+1}$). Ramification
from $\bp_k$ to $\bp_{k+1}$  implies that in going from $0\in
\prP^1_j$   (or 1) up to
$\bp_k$ there is no ramification. That is, there is a nontrivial element 
$\alpha\in
\PGL_2(\bC)$ and another cover $\phi_{\bp'_{k+1}}: X_{\bp'_{k+1}} \to \prP^1_z$ 
with the following
properties. 
\begin{edesc} \label{ramstops} \item $\phi_{\bp'_{k+1}}$ lies over 
$\phi_{\bp'_{k}}: X_{\bp'_{k}} \to
\prP^1_z$.
\item There is an isomorphism $\mu_k: X_{\bp_{k}} \to X_{\bp'_{k}}$ with 
$\phi_{\bp'_{k}}\circ \mu_k = \alpha\circ \phi_{\bp_{k}}$. 
\item There is {\sl no\/} such isomorphism $\mu_{k+1}: X_{\bp_{k+1}} \to 
X_{\bp'_{k+1}}$ lying over
$\mu_k$. 
\end{edesc}  

We refer to  \eqref{ramstops} as saying there is {\sl elliptic ramification\/} 
from
$\bar O_{k}$ to $\bar O_{k+1}$. If each $\bar M_4$ orbit $\bar O_{k+1}$ over 
$\bar O_{k}$ has no
points $\bp_k$ giving elliptic ramification, we say $\bar O_{k}$ has no elliptic 
ramification.  This is
exactly what happens in inner space examples from  
\cite[\S 8.1.1]{BFr} and with modular curves at level 0. 

\begin{guess} \label{noelliptic}  We assume the Main Conjecture hypotheses and 
that $k$ is large. Then, there
is no elliptic ramification above any component at level $k$. \end{guess}

The last sentence of 
Thm.~\ref{pdivis} assures  $G_k$ has no center. In turn this is equivalent to 
inner Hurwitz
space covers having  fine moduli. We test for reduced (inner) Hurwitz spaces to 
have fine moduli in 
two steps \cite[Prop.~4.7]{BFr}. Denote the pullback of $\bar O_k$ to 
$G\backslash
\ni(G_k,\bfC)$ by $\bar O_k^*$. Then, 
$\sH^\rd_{\bar O_k}$ has {\sl b-fine moduli\/} (fine moduli off the fibers over
$j=0$ and 1) if and only if  $\sQ''$ (\S\ref{shincidence}) is faithful on  $\bar 
O_k^*$. Given b-fine
moduli, then 
$\sH^\rd_{\bar O_k}$ has fine moduli if and only if neither $\gamma_0$ nor 
$\gamma_1$ has fixed points
on
$\bar O_k$. If $\sH^\rd_{\bar O_k}$  has b-fine (resp.~fine) moduli, then so 
does $\sH^\rd_{\bar O_{k+1}}$. 

Orbit shortening is another phenomenon that affects the Riemann-Hurwitz formula 
for computing the
genus of $\sH^\rd_{\bar O_k}$ components. \cite[Lem.~8.2]{BFr} explains orbit 
shortening as reducing the
length of
$q_2$ orbits in $\bar O_k^*$ to their image $\gamma_\infty$ orbits in $\bar 
O_k$. Though this is
significant for precise data about cusps, for it too there is a natural working 
hypothesis. 

\begin{guess} \label{noshortening} For $k$ large, there is no orbit shortening 
above $\sH^\rd_{\bar O_k}$.
\end{guess}
 
\subsection{$p$-growth of cusps} \label{pgrowth}  

We differentiate between two types of $\gamma_\infty$ orbits.

\subsubsection{$p$-divisible cusps} For $\bg=(g_1,g_2,g_3,g_4) \in O_i$, denote 
$|\lrang{g_2g_3}|$ by
$\mpr(\bg)$, the {\sl middle product\/} of $\bg$. Call 
$O_i$ (or $\bg$) {\sl
$p$-divisible\/}  if $p| \mpr(\bg)$.   Suppose $\bg' \in \bar O_{k+1}$
lies above $\bg$: $\bg' \mod M_k = \bg$. Express this with the notation 
$\bg'/\bg\in \bar O_{k+1}$. 
The $\bar M_4$ action guarantees the number of elements of $\bar O_{k+1}$ lying 
over
$\bg\in
\bar O_k$ depends only on $\bar O_k$. Denote $|\bar O_{k+1}|/|\bar O_k|$, the 
{\sl degree\/} of $\bar
O_{k+1}$ over $\bar O_k$,  by 
$[\bar O_{k+1},\bar O_k]$. 

Choose
one representative ${}_i\bg$ in each $O_i$,
$i=1,\dots, t$. Order the orbits so the first $t'$ of these, $\row O {t'}$, are 
the $p$-divisible
cusps. For $i\le t'$ and $\bg'/ {}_i\bg\in
\bar O_{k+1}$, Thm.~\ref{pdivis} implies $\mpr(\bg')=p\cdot \mpr(\bg_i)$. For 
$i\ge t'+1$, let 
$U_i$ be the number of
$p$-divisible $\bg'/ {}_i\bg\in
\bar O_{k+1}$. 

Assume the genus of $\bar O_k$ is 0 and the conclusion of
Conj.~\ref{noshortening} holds for $\bar O_k$: No orbits shorten from $\bar O_k$ 
to $\bar
O_{k+1}$. Then, the following holds 
\cite[Lem.~8.2]{BFr}: 
\begin{equation} \label{genbound} g_{\bar O_{k+1}} \ge \Bigl(\frac{(p-1)}{2p}t'-
1\Bigr)[\bar O_{k+1},\bar
O_k]+1 + \frac{(p-1)}2\sum_{t'+1\le i
\le t} U_i.\end{equation}  

The next proposition is essentially in \cite[Lem.~8.2]{BFr}. 
\begin{prop} \label{goup} To the previous assumptions add that the conclusion of
Conj.~\ref{noelliptic} holds for $\bar O_k$: There is no elliptic ramification 
from $\bar O_k$ to $\bar
O_{k+1}$. Then, $g_{\bar O_{k+1}}$ equals the right side of \eqref{genbound}. 
\end{prop}

\begin{exmp}[Situations where the genus rises] We keep the assumptions of 
Prop.~\ref{goup}. Then, the
genus rises if $t'> 2p/(p-1)$. If $t'=0$, expression \eqref{genbound} requires 
that some of the $U_i\,$s
are positive. So, we have $p$-divisible cusps on $\bar O_{k+1}$. At the next 
level that
forces $t'>0$. Notice with these assumptions that $[\bar O_{k+1},\bar
O_k]\ge p$, for there must be some ramification. There are a few boundary cases 
of concern, like $p=2$, 
$t'=4$ and $[\bar O_{k+1},\bar O_k]=2$.   
\end{exmp}

\section{Types of Schur multipliers} Continue notation from \S\ref{radlayers}. 
Suppose 
the first radical layer of 
$M_k$ contains $\one_{G_k}$. This means $G_k$ has a nontrivial $p$ part to its 
Schur multiplier.
Several  mysterious events can occur from this. 
\begin{edesc} \label{useSchur} \item  \label{useSchura} A $\bar M_r$ orbit $\bar 
O_k\!\subset\!
\ni(G_k,\bfC)^\rd$ may have nothing over it in
$\ni(G_{k+1},\bfC)^\rd$. \item  \label{useSchurb} Suppose $O$  is a 
$\gamma_\infty$ orbit in $\bar O_k$ that
is not
$p$-divisible. Still, it may have all
$\gamma_\infty$ orbits in $\ni(G_{k+1},\bfC)$ above it $p$-divisible.
\end{edesc} 
\cite[\S 9.3]{BFr} illustrates \eql{useSchur}{useSchura}, while \cite[\S 5.4--
5.5]{BFr} illustrates
\eql{useSchur}{useSchurb}. 

\subsection{Setting up for appearance of Schur multipliers} 
\label{schurMultStart}
Often we expect the
genus to go up dramatically with the level, even from level 0 to level 1, even 
if there are components of genus
0 at level 0. Example:
\cite[Cor.~8.3]{BFr} computes in the
$(A_5,\bfC_{3^4}, p=2)$ Modular Tower the genuses (12 and 9) of the two level 1 
components over the
one level 0 component of genus 0. The nontrivial Schur
multiplier of $G_0$  gives many
$p$-divisible cusps at level 1, though there are none at level 0. From the 
nontrivial Schur multiplier of
$G_1$ there comes a complete explanation of the two very different components at 
level 1. For example,
the genus 9 component has no component above it at level 2.  
We now show how to go beyond the spin cover situations that gave previous 
progress.    

Let $D$ be a $G_{k+1}$
submodule  of $M_{k+1}$ with    
\begin{triv} $M_{k+1}/D$ the trivial (1-dimensional) $G_k$ module $\one_{G_k}$. 
\end{triv} 
Assuming $G_0$ is $p$-perfect assures there is a {\sl unique\/} (up to
isomorphism) cover
$R_D=G_{k+2}/D \to G_k$ factoring through $G_{k+1}$ as a central extension 
having kernel $\bZ/p$
\cite[\S 3.6.1]{BFr}. Often we identify $D$ with the particular $\bZ/p$ quotient 
of $M_{k+1}$. We refer to $D$
or $R_D$ as a {\sl Schur quotient\/} at level $k$. 

\cite{BFr} sometimes uses computations in the characteristic
modules
$M_k$ in additive notation. That won't work in \S\ref{typeschur}; we consider 
these as subgroups
of $G_{k+1}$ and quotients of subgroups of $R_D$.

\subsection{Little $p$ central extensions} \label{typeschur} Use the previous 
notation for the characteristic
sequence
$\{M_k\}_{k=0}^\infty$ of $p$-Frattini modules. Let $\{P_k\}_{k=0}^\infty$ be a 
projective system of $p$-Sylows
of the corresponding groups 
$\{G_k\}_{k=0}^\infty$.
\S\ref{changeLevel1} noted that
$$M_0(P_k)=M_k=\ker(P_{k+1}\to P_k) \text{ for }k\ge 1,$$ though this may not 
hold for $k= 0$. Let
$D$ be an index
$p$ subgroup of the 1st Loewy layer of 
$M_{k+1}$ and let $R_D$ be the corresponding Schur quotient for $\tG p$:
$$\ker(R_D
\to G_{k+1})=\bZ/p \text{ with trivial $G_{k+1}$ action}.$$ Let $V^0_D \subset 
M_k$ be those nonidentity
elements of
$M_k$ with order
$p$ (rather than order $p^2$) lift to
$R_D$. Use $V_D$ to be $V_D^0$ augmented by the identity element of $M_k$ (its
lifts have order 1 or $p$). 

Denote the map $R_D\to G_{k}$ by $\phi_D$.  Use the notation $\hat M_D$
(resp.~$\hat V_D$) for the pullback of $M_k$ (resp.~$V_D$) in $R_D$. We will see 
that sometimes 
$V_D$ is not a group. Still, the notation is valid. Use 
$\hat m$ for some lift of  $m\in M_k$ to $\hat M_D$. Since
$\hat M_D$ is a central extension of $M_k$, $\hat m^p$ depends only on $m$, and 
not on $\hat m$. 
We 
display the context for groups of order $p^3$ appearing in our calculations.

\begin{itemize} \item $\bZ/p^2\xs \bZ/p=U_p$ with a generator of the right copy 
of $\bZ/p$ mapping 
$1\in \bZ/p^2$ to $1+ap$ for some $a$ prime to $p$. 
\item $(\bZ/p)^2\xs \bZ/p=W_p$ with a generator of the right $\bZ/p$ acting as 
the matrix 
$\smatrix 1 1 0 1$.   
\item $H_p$ is the small Heisenberg group : $2\times 2$ unipotent upper 
triangular matrices with every 
element of order $p$.

\end{itemize}  

\begin{lem} The group $H_p$ ($p>2$) is isomorphic to $W_p$. \end{lem}

\begin{proof} Note: $H_p$ has a normal
subgroup
$H'$ of index $p$. So $H'$ is $(\bZ/p)^2$. Further, the map $H_p\to 
H_p/H'=\bZ/p$ splits (every
element in
$H_p$ has order $p$).  Finally, with some choice of basis, $\smatrix 1 1 0 1$ 
has 
order $p$ (acting on $(\bZ/p)^2$). \end{proof}

The following lemma has notation appropriate for our applications. Still it is
very general: We could have $M_k$ can be any $\bZ/p$ module,
$\hat M_D$ any Frattini extension of it with $\bZ/p$ kernel and $V_D^{0}$ the 
elements of $M_k$ lifting to have
order $p$ in
$\hat M_D$.

\begin{lem} \label{liftingGroups} Consider pairs $m_1,m_2\in M_k$ where 
$\lrang{m_1,m_2}$ has 
dimension 2. Denote the group  generated by their lifts by $\lrang{\hat m_1,\hat 
m_2}$.  

Suppose $m_1,m_2\in V_D$. If $p=2$, then $\lrang{\hat m_1,\hat m_2}$ is either a 
Klein  4-group or 
it is $D_4$. If $p \ne 2$,then it is either $(\bZ/p)^2$; $U_p$; or it is $H_p$ 
and the commutator $(\hat m_1 ,\hat m_2 )$ generates the kernel of $\lrang{\hat 
m_1 ,\hat m_2}
\to \lrang{m_1 ,m_2}$.

Suppose $m_1,m_2\in M_k\setminus V_D$. If $p=2$, then $\lrang{\hat m_1,\hat 
m_2}$ is either 
$\bZ/2^2\times \bZ/2$ or it is $Q_8$ (the quaternion group). For $p\ne 2$, 
$\lrang{\hat m_1,\hat m_2}$
is $\bZ/p^2\times \bZ/p$ or it is $U_p$. 

If $m_1\in V_D^0$ and $m_2\in M_k\setminus V_D$,  
$\lrang{\hat m_1,\hat m_2}$ is either $\bZ/p^2\times \bZ/p$ or  $U_p$.  
\end{lem} 

\begin{proof} The cases for $p=2$ are in \cite[Lem.~2.24]{BFr}. Now consider the 
cases where $p$ is odd. 
If $m_1,m_2\in V_D^0$, then their lifts $ \hat m_1,\hat m_2$ have order $p$. 
Further, they either
commute or $(\hat m_1,\hat m_2)$ generates the kernel of $\lrang{\hat m_1,\hat 
m_2}\to \lrang{ m_1,
m_2}$. These properties determine the group to be the two cases in the 
statement. 

If $m_1,m_2\in M_k\setminus V_D$ then $ \hat m_1,\hat m_2$ have order $p^2$. If 
$ \hat m_1$ and 
$\hat m_2$ commute, the result is that in the lemma. Assume, however, they don't 
commute. Let
$H$ be an index $p$ (normal) subgroup of $\lrang{\hat m_1,\hat m_2}$. Suppose 
the natural map
$$\mu: \lrang{\hat m_1,\hat m_2}\to \lrang{\hat m_1,\hat m_2}/H$$ splits. Then, 
the group is either $U_p$ or
it is
$W_p=H_p$. The latter, however, has no elements of order $p^2$. So it is the 
former. We are done if we
show $\mu$ splits. Equivalently, with $C$ the center of $\lrang{\hat m_1,\hat 
m_2}$, we are done if some
nontrivial element of $\lrang{\hat m_1,\hat m_2}/C$ lifts to have order $p$ in
$\lrang{\hat m_1,\hat m_2}$. 

Conjugate $\hat m_2$ by $\hat m_1$. Since the quotient $\lrang{\hat m_1,\hat 
m_2}/C$ is abelian, 
the conjugate $\hat m_1 \hat m_2 \hat m_1^{-1}$ is $\hat m_2^{1+pa}$ for some 
integer $1\le a<p$.
Compute:   
$$A(\hat m_1,\hat m_2)\eqdef(\hat m_1\hat m_2)^p=(\prod_{i=1}^p(\hat m_1^i\hat 
m_2\hat 
m_1^{-i})^p \hat m_1^p=\hat m_2^{\sum_{i=1}^p (1+ap)^i}(\hat m_1)^p=\hat m_2^p 
\hat m_1^p.$$ 
Replace $\hat m_1$ by some $p'$ power of it to assure $A(\hat m_1,\hat m_2)$ is 
1, though
$\hat m_1$ and $\hat m_2$ still generate $\lrang{\hat m_1,\hat m_2}/C$. So $\hat 
m_1\hat m_2$ has order
$p$, giving the desired splitting. 

The last case works as the split case of the previous argument. 
\end{proof} 

Lem.~\ref{liftingGroups} differentiates between $p=2$ and general $p$. 

\begin{cor} \label{RDGk} Assume   
$\hat m_i$ is a lift of $m_i\in M_k\setminus V_D$ to $\hat
M_D$, $i=1,2$. Suppose too that $\lrang{m_1}\ne \lrang{m_2}$ and $V_D^0\cap 
\lrang{m_1,m_2}$ is
nonempty. Then,
$H_{m_1,m_2}=\lrang{\hat  m_1,\hat m_2}$ is isomorphic  to $\bZ/p^2\times  
\bZ/p$. 
\end{cor}

\begin{proof} The hypotheses say that  
$H_{m_1,m_2}$ has two generators of order $p^2$. It also has two generators with 
respective orders $p^2$
and  $p$. The resulting group has order $p^3$, and $(\bZ/p)^2$ as a quotient. 
Only  $\bZ/p^2\times
\bZ/p$ has these properties: elements  of order 4 (resp. $p^2$) don't generate 
the dihedral 
group (resp.~$U_p$); 
and the quaternion group's only element of order 2 is in its center.
\end{proof} 

 We use the following hypotheses to describe the possible groups $\hat M_D$ that 
actually can occur as
$\phi_D^{-1}(M_k)$ with $M_k$ the Frattini module for $G_k$. 
\begin{edesc} \label{modassume} \item \label{modassumea} 
For each $m_1 ,m_2\in M_k \setminus V_D$, with
$\lrang{m_1} \ne
\lrang{m_2}$ , $V^0_D \cap \lrang{m_1 ,m_2}\ne \emptyset$.
\item \label{modassumeb} Elements of $M_k \setminus V_D$ generate $M_k$.
\item \label{modassumec} $V_D$ is a submodule of $M_k$. \end{edesc}

\begin{prop} \label{liftMk} There always exists $\alpha_D \in M_k \setminus 
V_D$. 

Suppose \eql{modassume}{modassumea} and \eql{modassume}{modassumeb} hold. Then 
$\hat M_D$ is an abelian
group, and therefore a $\bZ/p^2 [G_k]$ module. Further, 
\eql{modassume}{modassumec} then holds and 
$\cup^{p-1}_{j=0} V_D \alpha_D^j=M_k$.

If $p=2$, then \eql{modassume}{modassumec} holds if and only if $\hat V_D$ is an 
abelian group isomorphic
to $V_D\times \bZ /2$. If both  \eql{modassume}{modassumea} and
\eql{modassume}{modassumec} hold, then $\hat M_D$ is a
$\bZ/4[G_k]$ module.
\end{prop}

\begin{proof} Suppose $ M_k \setminus V_D$ is empty. Then, a lift of any $m\in 
M_k$
(excluding the identity) to $R_D$ has order $p$. So, $R_D\to G_k$ (a Frattini 
extension) has kernel
of exponent $p$. Since, however, $G_{k +1} \to G_k$ is the universal exponent 
$p$-Frattini
cover of $G_k$, this shows $\hat M_D$ is a quotient of $M_k$. This contradiction 
produces
$\alpha_D \in M_k \setminus V_D$.

Now suppose \eql{modassume}{modassumea}. Cor.~\ref{RDGk} implies, if $m_1,m_2\in
M_k\setminus V_D$, then $H_{m_1,m_2}$ is an abelian group. If, further, 
\eql{modassume}{modassumeb} holds, then $\hat M_D$ has pairwise commuting 
generators. So, $\hat
M_D$ is an abelian group. For $\hat g$ a lift of $g\in G_k$ to $R_D$,  
$\hat g$ conjugation action  on  $\hat M_D$ depends only on 
$g$. This shows $\hat M_D$ is a $\bZ/p^2 [G_k]$ module.

Continue assuming \eql{modassume}{modassumea} and \eql{modassume}{modassumeb} 
hold. Then, for any
$\hat m$ lifting $m \in M_k$ to $R_D$, there is an integer $j$ with
$$((\hat m) (\hat\alpha_D)^{-j})^p=\hat m^p (\hat \alpha_D)^p=1.$$ Therefore $m
\alpha_D^{-j}\in V_D$ and $\lrang{\alpha_D}$  
fills out $M_k/V_D=\bZ/p$. 

Now assume $p =2$,  and \eql{modassume}{modassumea} and 
\eql{modassume}{modassumec} hold. 
If $m_1 \in M_k \setminus V_D$ and $v_2\in V_D^{0}$, then 
\eql{modassume}{modassumec} says $m_2=m_1v_2$. So, $H_{m_1,m_2}$ satisfies the
hypotheses of Cor.~\ref{RDGk} and is abelian. 
Consider  $H_{v_1,v_2}$ with $v_1,v_2\in V_D$. If this group has order larger 
than 4, then it
contains an order $4$ element. Thus, $\lrang{v_1,v_2}$ contains an element of 
$M_k\setminus
V_D$, contrary to \eql{modassume}{modassumec}. Conclude: For any $m_1,m_2\in 
M_k$,
$H_{m_1,m_2}$ is abelian. As elements $\hat m$ with $m\in M_k$ generate $\hat 
M_D$, this 
implies $\hat M_D$ is abelian.  
\end{proof}
 
\subsection{When $\hat M_D$ is abelian} \label{SchurQuotient} We now 
characterize when   
$\hat M_D$ is an abelian group (and so a $\bZ/p^2[G_k]$
module). This will play a big role in the expanded version of this paper. When 
this happens we call $D$ an {\sl
abelian\/}
$\bZ/p$ Schur quotient (of the level
$k$ Schur multiplier). The tool for this characterization of abelian Schur 
quotients starts with \cite[Prop.
9.6]{BFr}. We remind of the setup. 

Let $\psi_k: R_k\to G_k$ be the universal exponent
$p$ central extension of 
$G_k$. This exists from the $p$-perfect assumption \cite[Def. 3.18]{BFr}. Write 
$R_k$ as $\tG p/\ker_k^*$. 
Denote the closure of
$\lrang{(\ker_k,\ker_k^*), (\ker^*_k)^p}$ in $\ker_{k+1}$ by $\ker_{k+1}'$. 
Then, $\ker_{k+1}'$ defines $\tG
p/\ker_{k+1}'=R_{k+1}'$. 

A lift of $m \in
\ker(\phi_{k+1,k})$ to $\ker(R_{k+1}' \to G_k)$ has order $p^2$  if and only the 
image of $m$ in
$\ker(R_k\to G_k)$ is nontrivial. We now take advantage of how elements of $M_k$ 
with trivial images in
$\ker(R_k\to G_k)$ form a submodule. This is the heart of 
Prop.~\ref{abelianmod}, except we substitute a
$\bZ/p$ quotient $R_{D_{k-1}}$ of $R_k$ for $R_k$, etc. 

\subsubsection{Antecedents to a $\bZ/p$ Schur quotient} \label{antecedents}
Let $R_{D_{k\nm 1}}$ (resp.~$R_{D_{k}}$) be a Schur quotient at level $k$ 
(resp.~level$k+1$). 

\begin{defn} Suppose $\alpha
\in
\tG p$ generates $\ker(R_{D_{k\nm1}}\to G_{k})$ and 
$\alpha^p$ generates 
$\ker(R_{D_{k}}\to G_{k+1})$. Refer to
$R_{D_{k\nm 1}}$ as {\sl antecedent\/} to
$R_{D_{k}}$. \end{defn}  

\begin{prop} \label{abelianmod} In the notation above,  there exists $R_{D_{k\nm 
1}}$  antecedent to
$R_{D_{k}}$ if and only if the conditions of \eqref{modassume} hold for 
$R_{D_{k}}$. 
\end{prop}

\begin{proof}[Outline of Proof] Assume there exists $R_{D_{k\nm 1}}$  antecedent 
to
$R_{D_{k}}$. \cite[Prop. 9.6]{BFr}: $V_{D_k}$ consists of $m\in M_k$ that
map  trivially to $\ker(R_{D_{k\nm 1}}\to G_k)$; those $m$ that don't map 
trivially to
$\ker(R_{D_{k\nm 1}}\to G_k)$ form the nonidentity cosets in $M_k$ of $V_{D_k}$. 
Clearly these generate
$M_k$. That shows 
\eql{modassume}{modassumea}. Suppose $m_1,m_2\in M_k\setminus V_D$, and 
$\lrang{m_1,m_2}$ has rank 2. Since
the image of $\lrang{m_1,m_2}$ in $\ker(R_{D_{k\nm 1}}\to G_k)$ has rank 1, 
there must be elements of
$V_D^{0}$ in the kernel. This establishes 
\eql{modassume}{modassumeb}. 

Now consider the converse. Assume the conditions of \eqref{modassume} hold for 
$R_{D_{k}}$.
Prop.~\ref{liftMk} implies that  $V_D$ is a $\bZ/p$ module,  
so a $\bZ/p[G_k]$ module, having codimension 1 in $M_k$. So, $M_k/V_D$ is a 1-
dimensional $G_k$ module. In
notation from Prop.~\ref{liftMk}, choose $\alpha_D\in M_k \setminus V_D$ to have 
its image generate $M_k/V_D$. 

Suppose $M_k/V_D$ is not the trivial $G_k$ module. We can generalize \cite[Prop. 
9.6]{BFr} 
to suit any 1-dimensional $G_k$ module appearing in the 1st Loewy layer of 
$M_k$.  This would
produce $\alpha \in \tG p$ mapping to the image of $\alpha_D$ in 
$\ker(R_{D_{k\nm 1}}\to G_k)$ with the image
of  $\alpha^p$ generating  $\ker(R_{D_{k}}\to G_{k+1})$. Since, however, $G_k$ 
does not act trivially on
$\alpha_D \mod
\ker_k$, $G_{k+1}$ would not act trivially on $\alpha^p \mod \ker_{k+1}$. This 
contradicts that $\ker(R_{D_k}
\to G_{k+1})$ is the trivial $G_{k+1}$ module. 
\end{proof}

\subsubsection{The example $A_4$ and $p=2$} \label{VDs} For $p=2$, we stretch 
the
module theory discussion of
\cite[Ex.~9.2]{BFr}. As above, with   
$G_1=G_1(A_4)$, let $R_1$ be the universal exponent 2 central extension of 
$G_1$. 
Then, $\ker(R_1\to G_1)$ is a Klein 4-group 
$K_4$. We have $A_4=K_4\xs H$ with $H\simeq \bZ/3$ acting irreducibly on the 
$K_4$. This $H$ action 
extends to all the $M_k\,$s (see comment \S\ref{psplit}).  So, we differentiate 
between actions on the Klein
4-groups by writing write
$K_{4,H}$ for $K_4$ when the $H$ action is nontrivial. 

Consider the three distinct $\bZ/2$ Schur quotients $D_1$, $D_2$, $D_3$ for 
$G_1$. (The subscripts are
only decoration, not with the same meaning as in Prop.~\ref{abelianmod}.) One 
Schur quotient has antecedent the 
Schur multiplier of $A_4$ (at level 0). We identify this with
$D_1$.  The Loewy display for $\ker(G_1\to G_0)=M_0$ is $K_{4,H} \to 
K_{4,H}\oplus \one_{A_4}$
\cite[Cor.~5.7]{BFr}. 

We found the complete description of the $\hat M_{D_i}\,$s by finding three 
central
2-Frattini extensions of $M= K_{4,H}\oplus \one_{A_4}$, the 1st Loewy layer of 
$M_0$.  Then, we pulled these
back from the map $M_0\to M$, knowing that this would give all three possible 
Schur quotients. This makes it
clear the elements of the left most
$K_{4,H}$ lift to have order
$2$ in each of the
$\hat M_{D_i}\,$s. Here is the list.  
$$\begin{array}{rl}\hat M_{D_1} &= K_{4,H} \rightarrow K_{4,H} \oplus \bZ/4,\\ 
\hat M_{D_2} &= K_{4,H}  \rightarrow Q_8 \oplus \bZ/2,\\  
\hat M_{D_3} &= K_{4,H}  \rightarrow (Q_8 \cdot_{\bZ/2} \bZ/4).\end{array}$$
We have chosen to notation using a $\rightarrow$. It means
that quotienting out by an appropriate $\bZ/2$ center gives the corresponding 
Loewy display.

\subsection{Next steps in investigating Schur quotients} \label{nextsteps}
The followup paper will expand the application of  \cite[Prop.~9.8]{BFr} to 
$(G_0=A_5,
\bfC_{3^4}, p=2)$.  In this situation we refer to $\bg\in \ni(G_0,\bfC_{3^4})$ 
as a {\sl
perturbation\/} of an H-M rep.~if it has the form $(g_1,ag_1^{-1}a,bg_2b,g_2^{-
1})=\bg_{a,b} \in
\ni(G_1,\bfC_{3^4})^\rd$ and it lies over $$(g'_1,(g_1')^{-1},g'_2,(g_2')^{-
1})\in \ni(A_5,\bfC_{3^4})^\rd,
\text{ with } g_1'=(1\,2\,3) \text{ and }
g_2'=(1\,4\,5).$$  \cite[Prop.~9.8]{BFr} referenced H-M reps.~to find and 
characterize the resulting 
$\bar M_4$ orbits: $\bar O_1$ (genus 12) and $\bar O_2$ (genus 9).  We describe 
this in the
language above. 

Let
$R_{D_0}$ be the Frattini central extension of $G_1$ with antecedent the spin 
cover $\hat A_5=\SL_2(\bZ/5)$
of
$A_5=\PSL_2(\bZ/5)$. So, $D_0$ gives an abelian quotient. Then, with no loss, we 
may choose $a,b\in
M_0\setminus V_{D_0}$ for the representatives of form $\bg_{a,b}$ as above.  
There are two orbits for
the action of $A_5$ on
$M_0\setminus V_{D_0}$, called $M_3'$ (centralize elements of order 3) and
$M_5'$ (centralize elements of order 5). The 
16 elements $\bg_{a,b}\in \bar O_1$ (resp.~$\in \bar O_2$) have $a$ and $b$ in 
the same
(resp.~in different) conjugacy classes from
$M_3'$ and $M_5'$. 

We say in \S\ref{VDs} that the case $(G_k(A_4),\bfC_{3^4})$, with $k=1$, has a 
more complicated Schur
quotient  structure as $k$ rises. In the expansion of this paper we extend this 
to the
abelian cases at all levels of the $(A_4,\bfC_{3^4},p=2)$ and 
$(A_5,\bfC_{3^4},p=2)$ Modular Towers. We
know the number of Schur quotients of $G_k(A_4) $ rises with $k$. 
This is true for any {\sl
non-dihedral-like\/} split $G_0$ and any prime $p$ \cite{Darren}. 
\cite[\S5.7]{BFr} has a precise quote
when $p=2$ and a characterization of the phrase {\sl dihedral-like\/}.  We don't 
know yet how to extend it
to the nonabelian
$\bZ/2$ Schur quotients. We don't know yet if for $k\ge 1$,  
$G_k(A_5)$ [original paper had a typo, with $A_4$ in place of $A_5$] 
has nonabelian $\bZ/2$ Schur quotients, though we 
suspect so. [Added 9/01/04: Darren Semmen has shown for $k\ge 3$ there are 
nonabelian $\bZ/2$ Schur quotients. The 
argument  shows there
will be for any prime $p$ at suitably high levels for non-dihedral-like groups.] 
For all non-split
universal $p$-Frattini covers (assume $p$-perfect) it is [still] a mystery what to 
expect of their Schur
quotients.

\section{The diophantine goal} \label{dioph} It is instructive to see how the 
diophantine aspects of the
long studied modular curves work. I base my remarks on \cite[Chap. 5]{SeMW}. 

\subsection{Setup diophantine questions} Question \ref{naiveQuest2} is a
version of the Main Conjecture. When $r=4$ the \S\ref{MainConjTruth}  outline  
gives a good sense of it
holding. The tests are for how properties of the universal
$p$-Frattini cover of the finite group $G$ contributes to the conjecture 
holding. \cite{BFr} tested many
examples related to the case $G=G_0$ is an alternating group,
$p=2$ and the conjugacy classes have odd order. Details of those examples mirror 
what one
often finds in papers on modular curves. An extra complication is that levels of 
a Modular Tower
can have several components. Here's an example
problem that is tougher than it first looks. 

\begin{prob} \label{naiveQuest1} For {\sl each\/} number field $K$, find an easy 
argument for constructing 
a nonsingular curve  $X$ over $K$ with $0< |X(K)| < \infty$ (nonempty but 
finite). 
\end{prob}

Suppose $G_0$ is a $p$-perfect group, and $\bfC$ are $p'$ conjugacy
classes of $G_0$, such that all levels of the Modular Tower for $(G_0,\bfC,p)$ 
are nonempty.

\begin{prob} \label{naiveQuest2}  Give an elementary argument that for
any number field $K$, $\sH(G_k,\bfC)^\rd(K)=\emptyset$ for $k$ large.   
\end{prob} 
Let  $\bar \sH(G_k,\bfC)^\rd=\bar \sH_k$ be a nonsingular projective closure of 
$\sH(G_k,\bfC)^\rd$. Even
though
$\sH(G_k,\bfC)^\rd(K)=\emptyset$, there may rational points on the cusps of 
$\bar \sH(G_k,\bfC)^\rd$. 

\begin{prob} \label{naiveQuest3} Same hypotheses as Ques.~\ref{naiveQuest2}. 
Show, at suitably high
levels of the tower, all components have general type. 
\end{prob} 

We mean in Prob.~\ref{naiveQuest1}  to avoid such
heavy machinery as Falting's Theorem (proof of the Mordell Conjecture) or the 
Merel Theorem. We don't,
however, expect a completely trivial argument for a general number field $K$.  
The property states that number fields are not {\sl ample\/} (Pop's 
nomenclature). The goal is 
to produce a {\sl witnessing\/} nonsingular curve $X$ (given $K$) explicitly. 
This topic arose from
\cite[\S3.2.1]{DebBB}. We trace a set of ideas from Demjanenko and Manin as 
appropriate to our main point.
Then, we connect Problems 
\ref{naiveQuest1} and \ref{naiveQuest2}.

\subsection{Outline of key points of \cite{Manin}-\cite{Dem}} \label{mandem} The 
exposition from
\cite[Chap.~5]{SeMW} is convenient for this, especially for its review of 
effective aspects 
of Chabauty \cite{cha}. 

\begin{thm}[Chabauty] \label{chabauty} Suppose a curve $X$ generates an abelian 
variety $A$ and $\Gamma$ is
a finitely generated subgroup of $A(K)$ with $\rank(\Gamma) <\dim A=g'$. Then, 
$X(K)\cap \Gamma < \infty$. 
\end{thm}

\begin{proof}[Explicit aspects] Embed $K$ in a finite extension of $\bQ_p$ and 
regard it as that finite
extension. Assume $0_A\in X(K)$, and let $\omega_i$, $i=1,\dots,g'$, be a basis 
of holomorphic differentials
on $A$. These uniformize by $(\sO_L)^{g'}$ an open subgroup $U\subset A(L)$ in 
the $p$-adic topology via the
map
$P\in A(L) \mapsto \int_{0_A}^P(\row \omega {g'})$. With $\bar\Gamma$ the 
closure of $\Gamma$, suppose 
$X\cap\bar\Gamma$ is infinite. That gives a sequence of distinct points 
$P_i\mapsto P_0\in X\cap
\bar \Gamma$. Suppose $d<g'$ is the rank of the free group $\Gamma\cap U$. 
Change coordinates on $U$ so
the points $\gamma=(x_1,...,x_{g'})\in U\cap\bar\Gamma$ satisfy $x_1=0$ in a 
neighborhood of
$0_A\in X$. Then the analytic curve intersects $x_1=0$ in infinitely many 
points. This implies $x_1=0$ in
a neighborhood of $0_A$ in $X$. Since, however, $X$ generates $A$, the pullback 
of $dx_1=\omega_1$ is a
nontrivial holomorphic differential. So, it has at most $2g(X)-2$ zeros. 
\end{proof} 

\cite{col} uses Thm.~\ref{chabauty} to
effectively bound the number of points on some curves, an analysis that includes 
finding a bound
on $\rank(\Gamma)$ and bounding the number of torsion points of $A$ that might 
meet $X$. An effective Manin
Corollary comes down to effectively bounding the rank of
$\Pic^{(0)}(X_0(p^{k_0+1}))(K)$. The Weak-Mordell Weil Theorem gives such a 
bound \cite[p.~52, \S4.6]{SeMW}. 
Here, and for Modular Towers, we don't care that the bound does not precisely 
give the rank. 

The
Manin-Demjanenko argument reverts to Chabauty. We can see specific parameters we
must compute to make Thm.~\ref{chabauty} apply effectively. It starts by 
assuming
$X$ is any projective nonsingular absolutely irreducible variety over $K$, and 
$A$ is an abelian variety
over
$K$ for which there are morphisms $\row f m: X\to A$ defined over $K$. With 
$P_0\in X(K)$ assume $\bar
f_i=f_i-f_i(P_0)$ satisfy these properties. 
\begin{edesc} \label{MDhypoth} \item \label{MDhypotha} If $\sum_{i=1}^m n_i\bar 
f_i$ is zero on
$X$, $\row n m\in \bZ$, then 
$(\row n m)$ is 0.  
\item \label{MDhypothb} The rank of the divisor classes modulo algebraic 
equivalence on $X$ (the
N\'eron-Severi group) is one. 
\end{edesc} 
If $X$ is a curve, hypothesis \eql{MDhypoth}{MDhypothb} is automatic. Depending 
on your patience, the
following results are effective \cite[p.~63]{SeMW}. 

\begin{thm}[Conclusions of Manin-Demjanenko] If $m> \rank(A(K))$, then $X(K)$ is 
finite.
\end{thm} 

\begin{cor}[Manin] \label{manResult} If $k(p,K)$ is large, then 
$X_0(p^{k+1})(K)$ is finite. So,
we can choose $k(p,K)$ effectively dependent on bounding on 
$\Pic^{(0)}(X_0(p^{k_0+1}))$. \end{cor}
\begin{proof}[Key points] Choose $k_0$ so the genus of $X_0(p^{k_0+1})$ exceeds 
0. 
In parallel with the application to Modular Towers, \cite[Lem.~2.20]{BFr} 
computes the cusp widths of
$X_0(p^{k_0+1})$ (and so the growth of its genus) from a Hurwitz monodromy 
viewpoint. Let
$A=Pic^{(0)}(X_0(p^{k_0+1}))$. 

Assume the rank of
$\Pic^{(0)}(X_0(p^{k_0+1}))(K)$ is $m'$, and let $k(p,K)=k_0+m'+1$. We form the 
maps $f_i$,
$i=1,...,m'+1$. They all go from $X_0(p^{k_0+m'+1})$ to 
$\Pic^{(0)}(X_0(p^{k_0+1}))$. Represent points on
$X_0(p^{k_0+m'+1})$ by $(E,\lrang{\bp_{k_0+m'+1}})$, a 2-tuple consisting of an 
elliptic curve and a
$p^{k_0+m'+1}$ torsion point on it.  Let 
$\bp_i$ generate the cyclic subgroup of $\lrang{\bp_{k_0+m'+1}}$ of order
$p^i$. Then, 
$f_i$ maps the point $(E,\lrang{\bp_{k_0+m'+1}})$ to 
$(E/\lrang{\bp_{i}},\lrang{\bp_{k_0+1+i}}/\lrang{\bp_{i}}$
(then to $\Pic^{(0)}$ of $X_0(p^{k_0+1})$). 

Suppose $\omega$ is a nonzero holomorphic differential on $A$. Then these 
$f_i^*(\omega))$, $i=1,\dots,m'$, are
linearly independent over $\bC$. To see this take any cusp and express $\omega$ 
in a neighborhood of that cusp
as a power series
$f(q)dq$ with
$q=e^{2\pi i z}$, $z$ close to $i\infty$. The condition for analyticity of 
$\omega$ at the cusp is
that $f(q)$ is analytic and $f(0)=0$.  Now consider the pullbacks 
$\omega_i=f_i^*(\omega))$. The
argument of
\cite[p. 68]{SeMW} is that $f_i^*$ transforms $q$ to $q^{p^i}$; so the  
$\omega_i\,$s are linearly
independent as power series in the variable $q$ around the cusp.  
\end{proof}

\subsection{Using Demjanenko-Manin on Modular Towers} Each 
$X_0(p^{k+1})$ has cusps of unique widths that give rational points. This gives 
a relatively elementary
construction of curves that demonstrably have some, but only finitely many, 
rational points over $K$. This is
an acceptable answer to  Prob.~\ref{naiveQuest1}, though it is not so explicit a 
result as in 
\cite{col}.  In the expanded version of this paper we apply the argument of 
Cor.~\ref{manResult} to the
abelianization of a Modular Tower with
$r=4$ \cite[\S4.4.3]{BFr}.  It is crucial 
that the genus of tower components goes up (as outlined in
\S\ref{MainConjTruth}). This is motivation for developing explicit versions on 
that argument. 

An even greater potential occurs in using such an argument for $r\ge 5$. The
hypotheses of \eqref{MDhypoth} do not assume $X$ is a curve. An interesting 
consideration is
whether some version of \eql{MDhypoth}{MDhypothb} holds at high levels of a 
Modular Tower for a $p$-perfect
group $G=G_0$. We guess that at high levels the albanese varieties of (a 
nonsingular closure of) $\sH_k^\rd$
will not have a unique polarization. Are there are other approaches to this 
hypothesis?

\section{Finding projective systems of components over $\bQ$} 
\label{projsystems} Let
$\{G_k\}_{k=0}^\infty$ be the characteristic sequence of $p$-Frattini quotients 
of $\tG p$. Here is a
result from \cite{FrKMTIG}. 

\begin{prop} Let $r_0$ be any positive integer. Suppose each
$G_k$ has a
$K$  regular realization with $K$ a number field and no more than $r_0$ branch 
points. Then,
there exists a Modular Tower based on $(G_0,\bfC)$ with $r\le r_0$ conjugacy 
classes and
$\bfC$ a set of $p'$ conjugacy classes. \end{prop} 

If there are rational points at each level of a Modular Tower, then there is a 
projective system of 
(we always mean absolutely irreducible in this subsection)  $K$ components at 
each level.

\begin{defn} Call $\bfC$ {\sl gcomplete\/} if any subgroup meeting all conjugacy
classes  in it is all of $G$. It is {\sl $p$-gcomplete\/} if any subgroup 
meeting each $p'$
conjugacy class is automatically all of $G$.  \end{defn}

\begin{exmp} The group $A_5$ is 2-gcomplete. It is not, however, 3-gcomplete or 
5-gcomplete since it
contains $D_5$ and $A_4$. Further, if $G_0$ is $p$-gcomplete, then each of the 
characteristic
$p$-Frattini covers $G_k\to G_0$ is also $p$-gcomplete. \end{exmp} 

\begin{defn} Call $\bfC$  H-M-gcomplete (resp.~H-M-p-gcomplete) if upon removing 
any two
distinct inverse conjugacy classes $\C_i, \C_j$, the remaining conjugacy classes
$\bfC_{i,j}$ are gcomplete (resp.~p-gcomplete).\end{defn} 

We review a result from \cite[Thm.~3.21]{FrMT} on H-M reps.~(\S\ref{ramindex}). 

\begin{prop} Suppose at each level of a Modular Tower for the prime $p$, the H-M 
reps.~fall in one $H_r$
(Hurwitz monodromy) orbit $\bar O_{\rm HM}$. Then the component corresponding to 
$\bar O_{\rm HM}$ has
definition field $K$ where $K$ is the field of rationality of the conjugacy 
classes. This holds if the
conjugacy classes
$\bfC$  are $H-M-p$-gcomplete: Then there is a projective system of H-M 
components over $K$. \end{prop}

\begin{rem} As in Debes-Deschamps \cite{DebDes}, for $K$ a discrete valued 
field, you can get a result
over 
$K$ for a $p$-perfect group. For $g \in G$ of order $n$ call $d$ the {\sl 
cyclotomic order\/} of $g$
if $d$ is minimal for $K(\zeta_d)=K(\zeta_n)$. Suppose $\row {\C'} t$ is a $p$-
gcomplete set
of $p'$ conjugacy classes. Let $d_i$ be the cyclotomic order of elements of 
$\C'_i$, $i=1,\dots,t$. 
With $r=2\sum_{i=1}^t [K(\zeta_{d_i}):K]$, there are $r$ conjugacy classes 
$\bfC$ so the Modular Tower for
$(G,\bfC,p)$ has a projective system of $K$ points. \end{rem}

\begin{prob} Suppose given $(G,p)$ with $G$  a $p$-perfect group. Is there a 
$\bfC$ (consisting of $p'$
conjugacy classes) for which the Modular  Tower for $(G,\bfC,p)$ has a 
projective system of $\bQ$
components?\end{prob} 

When $r=4$ these gcomplete criteria fail to provide  projective systems  
of $\bar M_4$ orbits over $\bQ$. As an example problem consider this. 
\begin{prob} Does the Modular Tower for $(A_5,\bfC_{3^4},p=2)$ have a projective 
system of $\bQ$ (absolutely
irreducible) components?
\end{prob}

\section{Remaining topics from 3rd RIMS lectures in October, 2001} 
\label{revRIMS} 
Again, consider $r=4$, so Modular Towers are upper half plane quotients. Modular 
Towers has
practical formulas for analyzing how $G_\bQ$ acts on projective systems of 
tangential base points
attached to cusps. The talk concentration included these two topics which will 
appear in an expanded
version of this paper. 

\begin{edesc} \item Geometrically interpreting the $\sh$-incidence matrix 
attached to a level. 
\item  The
Ihara-Matsumoto-Wewers (\cite[App.~D.3]{BFr} and \cite{WeFM} modifying 
\cite{IharMat}) formulas for $G_{\bQ}$ 
acting on the fiber over tangential base points attached to Harbater-Mumford 
cusps (\S\ref{ramindex}). 
\end{edesc} Some cusps so stand out, that we recognize their presence at all 
levels of some Modular
Towers, even associated with specific irreducible components and having in their 
$\bar M_4$ orbits
other related cusps.  \cite[Prop.~7.11]{BFr} describes a natural pairing on 
cusps attached to 
Harbater-Mumford (H-M) representatives and near H-M representatives at {\sl 
all\/} levels of specific
Modular Towers, when $p=2$. An H-M rep. ~has a Nielsen class representative of 
form $(g_1,g_1^{-1},
g_2,g_2^{-1})$ (when $r=4$). A near H-M rep.~at level $k+1$ is a Nielsen class 
element that defines a real
point though it is not an H-M rep. From the formulas of \cite[\S6]{BFr}, a near  
H-M rep. has a specific
form related to that of its corresponding H-M rep. 

This applies to Schur quotients $R_{D_k}$ at level $k$ when $G_0=A_5$. This is a 
case when the Schur
quotient has an antecedent (\S~\ref{schurMultStart} and \S\ref{antecedents}). 
Since $R_{D_k}$ has a
nontrivial center, although the Hurwitz spaces $\sH(R_{D_k},\bfC)$ and
$\sH(G_{k+1},\bfC)$ are the same complex analytic spaces, the former is not a 
fine moduli space. We 
therefore expect there are fields $K$ and points $\bp\in \sH(R_k,\bfC)(K)$ with 
no
representing cover over $K$. Indeed, when $K=\bR$, the H-M reps.~give real 
points with representing covers
over $\bR$, while near H-M reps.~give real points for which the corresponding 
covers cannot have definition
field $\bR$ \cite[Prop.~6.8]{BFr}. In a later paper we will produce the $p$-adic 
version of
this result. This applies to our final topic.

\cite[App.~D]{BFr} discusses the Hurwitz space viewpoint of
Serre's Open Image Theorem, especially how 
one piece of Serre's Theorem extends to considering projective systems of 
Harbater-Mumford components. 
While speculative, this cusp approach goes around difficulties that appear in 
Serre's approach. 
Specifically:  It avoids a generic abelian variety hypothesis (death to most 
sensible applications), and the
cusp type replaces impossibly detailed analysis of many types of $p$-adic Lie 
Algebras. Serre extensively
used a $p$-Frattini property of the monodromy groups of modular curves over the 
$j$-line. A Galois-Ihara
problem arises when we inspect if this generalizes to each Modular Tower. The 
appearance of {\sl Eisenstein
series\/} and {\sl theta nulls\/} brings us to the present research line 
\cite[App.~B]{BFr}.

\providecommand{\bysame}{\leavevmode\hbox to3em{\hrulefill}\thinspace}

\end{document}